\documentclass[10pt]{article}

\usepackage{amsfonts,amsmath,amssymb,mathrsfs,dsfont,amsthm,natbib,authblk}
\RequirePackage[colorlinks,citecolor=blue,urlcolor=blue,linkcolor=blue]{hyperref}

\usepackage{xspace}
\RequirePackage[colorlinks,citecolor=blue,urlcolor=blue]{hyperref}
\usepackage{booktabs}

\newtheorem{theorem}{Theorem}
\newtheorem{lemma}{Lemma}
\newtheorem{proposition}{Proposition}
\theoremstyle{definition}
\newtheorem{remark}{Remark}
\newtheorem{example}{Example}

\newcommand{\ddr}{\mathrm{d}}
\newcommand{\Small}[1]{\textstyle #1 \displaystyle}
\newcommand{\R}{\mathbb{R}}
\newcommand{\fprime}{{f}}
\newcommand{\Indic}{\mathds{1}}
\newcommand{\E}{\mathrm{E}}
\newcommand{\edr}{\mathrm{e}}
\newcommand{\comillas}[1]{``\,#1\,''}

\allowdisplaybreaks

\makeatletter
\def\@maketitle{%
  \newpage
  \null
  \vskip 2em%
  \begin{center}%
  \let \footnote \thanks
    {\Large\bfseries \@title \par}%
    \vskip 2em%
    {\large
      \lineskip .5em%
      \begin{tabular}[t]{c}%
        \@author
      \end{tabular}\par}%
    \vskip .5em%
    {\normalsize \@date}%
  \end{center}%
  \par
  \vskip 1.5em}
\makeatother

\begin{document}

\title{Asymptotic behavior of the number of distinct values in a sample from the geometric stick-breaking process}

\author{Pierpaolo De Blasi%
  \thanks{\texttt{pierpaolo.deblasi@unito.it}}}
\affil{University of Torino and Carlo Alberto, Torino, Italy}

\author{Rams\'es H. Mena}
\affil{Universidad Nacional Aut\'onoma de Mexico, Mexico}

\author{Igor Pr\"unster}
\affil{Bocconi University and BIDSA, Milano, Italy}
\date{\today}



\maketitle

%
%


\begin{abstract}
Discrete random probability measures are a key ingredient of Bayesian nonparametric inferential procedures. A sample generates ties with positive probability and a fundamental object of both theoretical and applied interest is the corresponding random number of distinct values. The growth rate can be determined from the rate of decay of the small frequencies implying that, when the decreasingly ordered frequencies admit a tractable form, the asymptotics of the number of distinct values can be conveniently assessed. We focus on the geometric stick-breaking process and we investigate the effect of the choice of the distribution for the success probability on the asymptotic behavior of the number of distinct values. We show that a whole range of logarithmic behaviors are obtained by appropriately tuning the prior. We also derive a two-term expansion and illustrate its use in a comparison with a larger family of discrete random probability measures having an additional parameter given by the scale of the negative binomial distribution.
\end{abstract}

\\[-2mm]


\section{Introduction}
\label{sec:intro}

Discrete random probability measures can be represented by random frequencies at random locations as
\begin{equation}\label{eq:tildep}
  \tilde{p}(\ddr x)
  =\sum_{j\geq 1}w_j\delta_{x_j}(\ddr x).
\end{equation}  
The frequencies $(w_j)_{j\geq1}$ are $(0,1)$-valued variables such that $\sum_{j\geq 1}w_j=1$ almost surely (a.s.), and the locations $(x_j)_{j\geq 1}$ are draws from some distribution on a Polish space $\mathbb{X}$ endowed with the corresponding Borel $\sigma$-field. Discrete measures like $\tilde p$ are naturally suited to describe the structure a population made of potentially infinite different species or types, labeled by $x_j$, with certain random proportions modeled through $w_j$. Clearly, a sample drawn from $\tilde p$ will exhibit ties with positive probability  and thus the random number of distinct values in a sample of size $n$, here denoted by $K_n$, is of great interest. From a Bayesian nonparametric perspective the law of $\tilde p$ represents the prior distribution. Inference is carried out by predicting the number of new distinct values in an additional sample, conditional on an observed sample. See \citet{LMP07b}. According to the applied context at issue the distinct values or species are interpreted as distinct genes \citep{LMP07c}, words \citep{teh06}, economic agents \citep{lijoi2016}  etc. Another important statistical use of discrete random probality measures is in mixture modeling, when a layer is added to model the data distribution as in
  $$Y_i\sim f(Y_i|X_i),\quad 
  X_1,X_2,\ldots|\tilde p 
  \overset{\text{iid}}{\sim}\tilde p$$ 
for some probability kernel $f(y|x)$. Here $\tilde p$ acts as mixing distribution and $K_n$ represents the random number of mixture components, thus providing a flexible way to model unobserved heterogeneity in the population. The mixture is characterized by the component distribution $f(y|x_j)$, usually referred as the $j$th mixture component, and the mixing weights $w_j$. See \cite{IEEE15} for a recent review on the inferential implications of different choices of $\tilde p$. 

In probability theory, distributional properties of $K_n$ are of prime interest in combinatorial stochastic processes; see e.g. \cite{Arr:Tav:Bar:03},  \cite{Pit:06}, \cite{Gne:10}, \cite{Gne:Han:Pit:07}. The techniques employed to study the law of $K_n$ depend on the construction of the random frequencies $(w_j)$. \cite{Kar:67} studied the case of fixed frequencies and derived a key result, which forms the basis to establish general strong laws for $K_n$: it states that the growth of $K_n$ is ultimately determined by how small the small frequencies are, which can be conveniently expressed by the {\it tail behavior} of $(w_j)$ once decreasingly ordered. In particular, the faster the decay to zero, the slower $K_n$ diverges to infinity as $n$ increases. There exist essentially two regimes, logarithmic and polynomial growth. Notable examples are, respectively, the Dirichlet process \citep{Fer:73} and its two parameter extension known as Pitman-Yor process \citep{Pit:Yor:97}. 
The associated distributions of the frequencies in decreasing order, termed Poisson-Dirichlet and the two-parameter Poisson-Dirichlet, respectively, are not tractable enough for a direct application of Karlin's theory. Instead, the distribution of $K_n$ is derived from the Ewens and the Pitman-Ewens sampling formulae, cf. \cite{Pit:06}. In the former case $K_n$ is asymptotically normal, with both mean and variance of the order $\log n$. In the latter case the scale of $K_n$ is $n^\alpha$, where $\alpha\in(0,1)$ is the discount parameter of the Pitman-Yor process. The logarithmic growth of the Dirichlet process was first pointed out in \cite{Kor:Hol:73}. Within the logarithmic regime growth behaviors of $K_n$ slower than the logarithm, e.g. $(\log n)^\alpha$ with $\alpha<1$ or even iterated logarithms, can be achieved with so-called hierarchical processes \cite{camerlenghi2019}; see also \cite{argiento,bassetti2020}.
In this paper we are able to identify models leading to growth rates of the type $(\log n)^\beta$ with $\beta>1$, specifically we establish a growth rate $(\log n)^{m+2}$, for $m$ a nonnegative integer, for a class 
of tractable class of discrete random probability measures.
We stress that from a modeling perspective it is crucial to have tractable models, which cover the whole range of possible growth rates. See \cite{LMP07a,IEEE15,dahl,caron2017,francois2019,DCT20} for motivation and discussion of these issues in diverse application contexts also beyond exchangeability.
Note that a power logarithmic growth of $\E(K_n)$ can be obtained by means 
of a Dirichlet prior with a somehow artificial sample size-dependent 
specification of the total mass parameter; in particular, one needs the 
total mass parameter to grow with n, which leads to an increasingly 
informative prior as more data becomes available, an unnatural scenario.

The asymptotic evaluation $K_n$, together with its limiting distribution, has been also object of extensive research in the context of regenerative composition structures; see \cite{Gne:10} for a survey. In this setting the frequencies $(w_j)$ are constructed from the range of a multiplicative subordinator, that is from the exponential transform $1-\exp\{S(t)\}$ of a subordinator $S(t)$. The logarithmic and polynomial regimes can be recovered from Karlin's theory according to the variation at zero of the right tail of the L\'evy measure of $S(t)$. The $\log n$ regime corresponds to finite L\'evy measures, that is when $S(t)$ is a compound Poisson process. In this case, the frequencies $(w_j)$ can be conveniently defined in terms of a stick-breaking procedure, or residual allocation scheme, with
\begin{equation}\label{eq:SB}
  w_j=W_j\Small{\prod_{\ell<j}}(1-W_\ell)
\end{equation}  
for $(W_\ell)_{\ell\geq 1}$ independent and identically distributed (iid) $(0,1)$-valued random variables with distribution determined by the L\'evy measure. Exploiting the renewal representation of the composition structure, aymptotics for the moments of $K_n$ and a central limit theorem can be derived; cf. \cite{Gne:04}, \cite{Gne:etal:09}. 
\cite{Gne:Pit:Yor:06aop} show that when the right tail of the L\'evy measure is regularly varying at zero with index $-1< \alpha<0$, the scale of $K_n$ is $n^\alpha$ and the partition structure induced by the Pitman-Yor process can be recovered \citep{Gne:Pit:05aop}. In contrast, when the right tail diverges at zero like a slowly varying function, e.g. for $S(t)$ a gamma subordinator, a central limit theorem with mean of the order $(\log n)^2$ and variance of the order $(\log n)^3$ is obtained, cf. \cite{Gne:Pit:Yor:06ptrf}.

Discrete random probability measures with $w_j$ as in \eqref{eq:SB}, not necessarily with identically distributed $(W_\ell)_{\ell\geq 1}$, have been proposed in \cite{Ish:Jam:01} as a Bayesian nonparametric model and termed stick-breaking priors. The Dirichlet and the Pitman-Yor processes belong to this class, their distinctive property being that the law of $(w_j)_{j\geq 1}$ is invariant under size-biased permutation. 
In this setting, the distribution of $W_1$ is called the {\it structural distribution} of $(w_j)$, and the limiting behavior of $K_n$ in the Dirichlet and the Pitman-Yor process cases can be also derived using Karlin's theory from the variation at zero of the structural distribution; see \cite{Gne:Han:Pit:07}.

In this paper we further broaden the realm of application of the fundamental result of Karlin in order to derive a two-term expansion of the mean of $K_n$. 
The expansion relies on de Haan's regular variation theory and requires a precise assessment of the tail behavior of $(w_j)$ together with a deconditioning argument, cf. Theorem \ref{th:2nd-order}. To illustrate the applicability of this technique, we consider the geometric stick-breaking process, first proposed in \cite{Fue:Men:Wal:10}, which gained quite some popularity in Bayesian applications \citep{Men:Rug:Wal:11,Gut:Gut:Men:14,HMNW16}.
It is a discrete random probability measure \eqref{eq:tildep} with locations independent of the frequencies and, importantly,  $(w_j)_{j\geq1}$ naturally arranged in decreasing order, which facilitates the evaluation of the tail behavior of the sequence. Specifically the frequencies are of geometric type,
\begin{equation}\label{eq:geom}
  w_j=p(1-p)^{j-1},\quad j=1,2,\ldots
\end{equation}
with $p$, the probability of success, random and endowed with a (prior) distribution $\pi(p)$ on $(0,1)$. 
In Theorem \ref{th:pier2} we derive a two-term expansion for a choice of $\pi(p)$, the key technical tool being the regular variation of fractional integrals.
As anticipated, the leading term shows that the mean of $K_n$ can covers the whole range of logarithmic behaviors $(\log n)^{m+2}$, for $m$ a nonnegative integer, upon setting $\pi(p)$ as an exponential transform of the gamma distribution of shape parameter $m$. From a practical perspective this result widens the range of achievable asymptotic behaviors by means of tractable models and also allows  a principled prior elicitation.  
To illustrate the importance of the second order term in the expansion, we also consider an extension of the geometric stick-breaking process, which has an additional parameter $s$ corresponding to the scale of the negative binomial distribution. 
Such a construction reduces to the geometric stick-breaking process when $s=2$ and was exploited by \cite{Deb:Mar:Men:Pru:20} within a mixture model, to which the present study provides further theoretical support.
The frequencies $(w_j)_{j\geq1}$ are still decreasingly ordered and are available in closed form for any integer $s\geq 2$. The parameter $s$ determines the tail behavior of $(w_j)_{j\geq1}$, the larger $s$ the slower the decay to zero. In order to single out the effect of $s$ on $K_n$, we set $s=3$ and compare the asymptotic behavior of the mean of $K_n$ with that of the geometric stick-breaking case, while keeping $\pi(p)$ to be uniform. It turns out that $K_n$ grows faster for $s=3$, as predicted by Karlin's theory, the difference however emerging only in the second order term of the expansion, cf. Proposition \ref{prop:pier3}. 
We conjecture that similar conclusions apply also for $s$ an integer larger than $3$  and other prior specifications of $\pi(p)$, although we do not pursue it here. It would be of interest to investigate the asymptotics of higher order moments like the variance and whether a central limit theorem holds. These are left for future research.
\medskip

{\it Layout of the paper.}$\quad$ In Section \ref{sec:2} we review Karlin's theory and establish a general two-term expansion of the mean of $K_n$. In Section \ref{sec:3} we introduce the geometric stick-breaking process and investigate the impact of the choice of prior $\pi(p)$ on the asymptotic behavior of $K_n$. In Section \ref{sec:4} we deal with the negative binomial extension and apply the asymptotic expansion of Section \ref{sec:2} to show that the scale parameter $s$ enters in the second order term. Some proofs are deferred to the \hyperref[appn]{Appendix}.
\medskip

{\it Notation.}$\quad$
Let $F(x)$ be a positive nondecreasing function on $\R$ with $F(x)=0$ for $x\leq 0$ and $\alpha\geq 0$. The fractional integral of order $\alpha$ of $F(x)$ is given by
\begin{equation*}
  {}_\alpha F(x)=\frac{1}{\Gamma(\alpha+1)}
  \int_0^x(x-t)^\alpha \fprime(t)\ddr t.
\end{equation*}  
We use $f\sim g$ for $f/g\to1$, the limit being clear from the context. When either $f$ or $g$ is random, the notation $f\sim_\text{a.s.} g$ means that the asymptotic relation holds with probability one. For $x$ a real number, $\lfloor x \rfloor$ is the integer part of $x$.

\section{Occupancy problem and regular variation}
\label{sec:2}

Let $\tilde p$ be a discrete random probability measure \eqref{eq:tildep}. Assume $(w_j)_{j\geq1}$ and $(x_j)_{j\geq 1}$ are independent with $(x_j)_{j\geq 1}$ independent and identically distributed form a non atomic distribution. Then  $\tilde p$ is a {\it species sampling model} \citep{Pit:95}. The partition induced by a sample from $\tilde p$ depends only on  the random frequencies $(w_j)_{j\geq1}$ and can be studied in terms of a multinomial occupancy problem. The theory is well established and dates back to the seminal paper \cite{Kar:67}. The main tools are a Poissonization argument and regular variation theory. We provide a concise overview taking the set-up from \cite{Gne:Han:Pit:07}.

The multinomial occupancy problem can be described as the experiment of throwing balls independently at a fixed infinite series of boxes, with probability $w_j$ of hitting the $j$th box. First consider the case of fixed, or non random, frequencies. As $n$ balls are thrown, their allocation is captured by the array $X_n=(X_{n,j})_{j\geq 1}$ where $X_{n,j}$ is the number of balls out of the first $n$ that fall in box $j$. $K_n$, the number of occupied boxes, is then given by
  $K_n=\sum\nolimits_{j\geq 1}\Indic(X_{n,j}>0)$
with mean
  $$\E(K_n)=\sum\nolimits_{j\geq 1}(1-(1-w_j)^n).$$
In general, it is difficult to work with $\E(K_n)$ since the indicators in $K_n$ are not independent. In the Poissonized version of the problem the balls are thrown in continuous time at epochs of a unit rate Poisson process $(P(t),t\geq 0)$, which is independent of $(X_n,n=1,2,\ldots)$. The balls then fall in the boxes according to independent Poisson processes $(X_j(t))_{t\geq 0}$, at rate $w_j$ for box $j$. Hence
  $K(t):=K_{P(t)}
  =\sum\nolimits_{j\geq 1}\Indic(X_j(t)>0)$
and
  $$\Phi(t):=\E(K(t))
  =\sum\nolimits_{j\geq 1}(1-\edr^{-tw_j}).$$
Encoding the frequencies into the counting measure $\nu(\ddr x)=\sum_{j\geq 1}\delta_{w_j}(\ddr x)$ and integrating by parts, 
\begin{equation*}
  \Phi(t)=\int_0^1(1-\edr^{-tx})\nu(\ddr x)
  =t\int_0^1\edr^{-tx}
  \overrightarrow\nu(x)\ddr x,
\end{equation*}  
where 
  $\overrightarrow\nu(x)=\nu([x,1)),$
the right tail of $\nu$, represents the number of frequencies $w_j$ not smaller than $x$.
$\Phi(t)$ provides an approximation of $\E(K_n)$ for $n$ large according to
\begin{equation}\label{eq:Poissonization}
  |\E(K_n)-\Phi(n)|
  \leq \Small{\frac2n}\Phi(n)\to 0
\end{equation}  
cf. \cite[Lemma 1]{Gne:Han:Pit:07}. The convenience of working with $\Phi(t)$ is that, being $\Phi(t)$ the Laplace-Stieltjes transform of $\overrightarrow\nu(x)$, its behavior as $t\to\infty$ is determined by the behavior of $\overrightarrow\nu(x)$ as $x\to 0$ by an application of the Tauberian theorem; see \cite{Bin:Gol:Teu:87} for a full account on Abel-Tauberian theorems for Laplace transforms. Hence, ultimately, by regular variation theory the growth of $\E(K_n)$, as $n\to\infty$, is determined by the behavior of $\overrightarrow\nu(x)$ at zero. In the case of random frequencies, the same result holds with the counting measure $\nu(\ddr x)$ being replaced by its mean measure, and correspondingly adapting the meaning of $\overrightarrow\nu(x)$. See \cite[Section 7, Page 162]{Gne:Han:Pit:07} and Section \ref{sec:3} for an illustration. 

Here we work under the hypothesis that $\overrightarrow\nu(x)$ is slowly varying at zero, that is
  $\lim_{x\to 0}\overrightarrow\nu(\lambda x)/\overrightarrow\nu(x)=1$
for all $\lambda>0$. According to 
\cite[Theorems 1.7.1' and 1.7.6]{Bin:Gol:Teu:87} (see also \cite[Proposition 19]{Gne:Han:Pit:07}), 
  $\Phi(1/x)\sim\overrightarrow\nu(x)$
as $x\to0$, so that via \eqref{eq:Poissonization} \begin{equation*}
  \E(K_n)\sim 
  \overrightarrow\nu\big(\Small{\frac{1}{n}})
  \quad\mbox{as }n\to\infty,
\end{equation*}  
cf. \cite[Theorem 1']{Kar:67}. In Theorem \ref{th:2nd-order} we derive a two term expansion of $\E(K_n)$ under the hypothesis that $\overrightarrow\nu(x)$ is a de Haan slowly varying function at zero, that is for a constant $c$ and a slowly varying function $\ell(x)$ at zero, called the auxiliary function of $\overrightarrow\nu(x)$, 
\begin{equation}\label{eq:deHaan}
  \frac{\overrightarrow\nu(\lambda x)-\overrightarrow\nu(x)}
  {\ell(x)}\to c\log \lambda,
  \quad\text{as }x\to0.
\end{equation}  
%
%
\begin{theorem}\label{th:2nd-order}
If $\ell(x)$ is slowly varying at zero and $c\geq 0$ satisfy \eqref{eq:deHaan} for all $\lambda>0$, then
\begin{equation*}
  \E(K_n)=\overrightarrow\nu(1/n)-c\gamma\ell(1/n)+o(\ell(1/n)),
  \quad \text{as }n\to \infty
\end{equation*}  
where $\gamma$ is the Euler-Mascheroni constant. 
\end{theorem}
The proof consists in an adaptation  to the present setting of \cite[Theorem 3.9.1]{Bin:Gol:Teu:87} for the study of the remainder of Tauberian theorem, $\Phi(1/x)-\overrightarrow\nu(x)$, as $x\to 0$, combined with an application of \eqref{eq:Poissonization}. The proof is reported in the \hyperref[appn]{Appendix}.
In order to apply this result, one needs to establish the variation of $\overrightarrow\nu(x)$ at $0$, so some explicit or at least tractable form of $\overrightarrow\nu(x)$ is in order. In the next two sections we apply the asymptotic expansion of Theorem \ref{th:2nd-order} to species sampling priors that features stochastically decreasing frequencies for which $\overrightarrow\nu(x)$ is tractable enough.

\section{Geometric stick-breaking process}
\label{sec:3}

The geometric stick breaking process is a species sampling model with random frequencies $(w_j)_{j\geq1}$ of geometric type,
\begin{equation*}
  w_j=p(1-p)^{j-1},\quad j=1,2,\ldots
\end{equation*}
with random success probability $p$. 
The number of frequencies $w_j$ not smaller than $x$,  
  $\max\{j:\ p(1-p)^{j-1}\geq x\}$,
can be explicitly found as the solution in $j$ to the equation $p(1-p)^{j-1}=x$. By direct calculation
\begin{equation*}
  \overrightarrow\nu(x,p)=\bigg\lfloor 
  \frac{\log(x/p)}{\log (1-p)}+1\bigg\rfloor\,
  \Indic_{(p\geq x)},
\end{equation*}  
where the notation $\overrightarrow\nu(x,p)$  makes the dependence on $p$ explicit. The case of fixed $p$ provides an illustration of Theorem \ref{th:2nd-order}.
%
%
\begin{example}
Let $K_n$ be the number of distinct values among $n$ iid draws from the geometric distribution with success probability $p$. Accurate formulae for the mean and the variance of $K_n$ are given in \cite{Arc:etal:06}. 
Since $\overrightarrow\nu(x,p)\sim \log x/\log(1-p)$ as $x\to 0$, $\overrightarrow\nu(x,p)$ is a de Haan slowly varying function with auxiliary function $\ell(x)=1$ and $c=1/\log(1-p)$, cf. \eqref{eq:deHaan}. Hence Theorem \ref{th:2nd-order} yields
\begin{equation*}
  \E(K_n)=
  \bigg\lfloor\frac{\log (np)}{|\log(1-p)|}
  +1\bigg\rfloor
  +\frac{\gamma}{|\log(1-p)|}+o(1)
  \quad\mbox{as }n\to\infty,
\end{equation*}  
in accordance with the expansion of \cite[Theorem 1]{Arc:etal:06}.
\end{example}
Now return to the random case with $\pi(p)$ on $(0,1)$ denoting the (prior) distribution of the success probability $p$ in \eqref{eq:geom}. The results about the expected value of $K_n$ now hold with $\nu(\ddr x)$ being the mean measure of the counting measure $\sum_{j\geq 1}\delta_{w_j}$ and $\overrightarrow\nu(x)$ obtained by averaging the number of frequencies $w_j$ not smaller than $x$ with respect to $\pi(p)$:
  $$\overrightarrow\nu(x)
  =\int_0^1\overrightarrow\nu(x,p)\pi(p)\ddr p.$$
In the sequel it is convenient to work with 
\begin{equation*}
  m(x)
  =\int_x^{1}\frac{\log x-\log p}{\log(1-p)}
  \pi(p)\ddr p,
\end{equation*}  
since
  $m(x)\leq \overrightarrow\nu(x)
  \leq m(x)+1$. 
The variation of $\overrightarrow\nu(x)$ in zero can then be studied in terms of $m(x)$.
By the change of variable $t=\log 1/p$,
\begin{equation}\label{eq:m(x)}
  m(x)
  =\int_0^{\log 1/x}
  \bigg(\log \frac{1}{x} - t\bigg)
  \pi(\edr^{-t})\fprime(t)\ddr t,\quad
  \fprime(t)
  =\frac{\edr^{-t}}{-\log(1-\edr^{-t})}
\end{equation}   
Properties of $\fprime(t)$ in \eqref{eq:m(x)} are collected in Lemma \ref{lemma:fprime}, whose proof is deferred to the \hyperref[appn]{Appendix}.
%
%
\begin{lemma}\label{lemma:fprime}
The function $\fprime(t)$ defined in \eqref{eq:m(x)} is nondecreasing on $\R_+$ with $\lim_{t\to 0}\fprime(t)=0$, $\lim_{t\to\infty}\fprime(t)=1$ and 
  $1-\fprime(t)\sim\edr^{-t}/2$ as $t\to\infty$.
Moreover
  $\int_0^\infty(1-\fprime(t))\ddr t=\gamma$, 
with $\gamma=-\Gamma'(1)-\int_0^\infty(\log x)\edr^{-x}\ddr x$ the Euler-Mascheroni constant.  
\end{lemma}

The variation at zero of $m(x)$ is determined by $\fprime(t)$ and the success probability distribution $\pi(p)$. First consider $p$ uniformly distributed on the unit interval. 
%
%
\begin{proposition}\label{prop:pier}
Let $p$ in \eqref{eq:geom} be uniformly distributed on $(0,1)$. Then
\begin{align*}
  \overrightarrow\nu(x)&=
  \frac{1}{2}\big(\log 1/x\big)^2
  -\gamma\log1/x+O(1),\quad x\to 0\\
  \E(K_n)&=\frac12 (\log n)^2
  +o(\log n),
  \quad n\to\infty.
\end{align*}  
\end{proposition}
%
%
\begin{proof}
For $\pi(p)=\Indic_{(0,1)}(p)$, $m(x)$ in \eqref{eq:m(x)} is given by
  $$m(x)=\int_0^{\log 1/x}
  (\log 1/x -t)\fprime(t)\ddr t
  ={}_1F(\log1/x),$$
where $F(t)=\int_0^t\fprime(s)\ddr s$  and 
  ${}_1F(x)=\int_0^x(x-t)\fprime(t)\ddr t$
is the fractional integral of order one of $F$. To prove the first statement,  it is sufficient to prove it for $m(x)$ in place of $\overrightarrow\nu(x)$. Integrating by parts,
  ${}_1F(x)=\int_0^xF(t)\ddr t$.
Hence, since $\log1/x\to\infty$ as $x\to 0$, we derive an asymptotic expansion of $F(x)$ as $x\to\infty$. According to Lemma \ref{lemma:fprime}, $\fprime(x)$ is a distribution function on $\R_+$. 
Moreover, given 
  $1-\fprime(t)\sim \edr^{-t}/2$ as $t\to\infty$,
the distribution function $\fprime(x)$ has moments of any order and, in particular, the first moment is equal to the Euler-Mascheroni constant $\gamma$. 
Then, $F(x)$ is regularly varying at infinity with exponent $\beta=1$ and, as $x\to \infty$, 
\begin{equation}\label{eq:F(x)}
  F(x)=x-\int_0^x(1-\fprime(t))\ddr t
  =x-\gamma+\int_x^\infty(1-\fprime(t))\ddr t
  =x-\gamma+O(\edr^{-x}).
\end{equation}  
Computing $\int_0^x F(t)\ddr t$ with the asymptotic expansion $F(x)\sim x-\gamma+O(\edr^{-x})$ leads to
  $${}_1F(x)
  =\int_0^xF(t)\ddr t
  =\frac{(x-\gamma)^2}2+O(1),\quad x\to\infty.$$
Substituting $x$ for $\log 1/x$ yields the first statement. In view of the application of Theorem \ref{th:2nd-order}, note that, as $x\to 0$,
\begin{align*}
  m(\lambda x)-m(x)
  &=\frac12\big(
  (\log(\lambda x))^2-(\log x)^2
  \big)
  +\gamma\big(\log (\lambda x)-\log x)+O(1)\\
  &=\frac12\big((\log x)^2+2\log\lambda\log x
  -(\log x)^2\big)+O(1)
  =\log\lambda \log x+O(1)
\end{align*}
so that, as $x\to 0$,
  $(m(\lambda x)-m(x))/\log x\to \log\lambda$.
Hence $\overrightarrow\nu(x)$ is a de Haan slowly varying function at zero with auxiliary function $\ell(x)=\log x$ and $c=1$, cf. \eqref{eq:deHaan}. An application of Theorem \ref{th:2nd-order} yields the second statement.
\end{proof}
%
%
\begin{remark}
Using only the leading term of the expansion of $m(x)$ in the application of Theorem \ref{th:2nd-order}, we would get the second order term in the asymptotic expansion of $\E(K_n)$ wrong, i.e. differing by $\gamma\log n$. Hence, in this case, an application of Karamata's Theorem to the evaluation of $\int_0^x (x-t)\fprime(t)\ddr t$ would be not precise enough, as the latter would yield
  $\int_0^x (x-t)\fprime(t)\ddr t
  \sim\frac12xF(x)$
and, in turn,
  $m(x)\sim\frac12(\log\frac1x)^2$.
\end{remark}

%
%
Next we tackle the case of the success probability distribution $\pi(p)$ chosen such that the integrand in \eqref{eq:m(x)} behaves like $t^m\fprime(t)$ for $m$ a positive integer. This is obtained by setting $p=\edr^{-X}$ for $X\sim\text{ga}(m+1,1)$, a gamma distributed random variable with shape $m+1$ and unit rate. Note $m=0$ yields the uniform distribution considered in Proposition \ref{prop:pier}. As detailed in Theorem \ref{th:pier2} below, this choice makes $\E(K_n)$ grow as a power of $\log n$ with exponent determined by $m$. Before that, we first provide an illustration of how the arguments used in Proposition \ref{prop:pier} can be adapted to the case $m=1$, paving the way for the techniques used in the general case.
%
%
\begin{example}\label{example:1}
By direct calculation, the density function of $p\stackrel{d}{=}\edr^{-X}$, for $X\sim\text{ga}(2,1)$, is $\pi(p)=-\log p$. Then $m(x)$ in \eqref{eq:m(x)} is given by 
  $$m(x)=\int_0^{\log 1/x}(\log 1/x-t)t\fprime(t)\ddr t  
  =\int_0^{\log 1/x}\int_0^ts\fprime(s)\ddr s\,\ddr t.$$
where the second equality follows by integration by parts. We first derive an asymptotic expansion for $\int_0^xt\fprime(t)\ddr t$ as $x\to \infty$. Since
  $$\int_0^x t\fprime(t)\ddr t
  =xF(x)-\int_0^x F(t)\ddr t
  =xF(x)-{}_1F(x)$$
using the asymptotic expansion 
  $F(x)=x-\gamma+O(\edr^{-x})$
in \eqref{eq:F(x)} 
we find
  $$\int_0^x t\fprime(t)\ddr t
  =x(x-\gamma)+O(x\edr^{-x})
  -\frac{(x-\gamma)^2}2+O(1)
  =\frac{x^2}2+O(1)$$
so that
  $$
  \int_0^x\int_0^ts\fprime(s)\ddr s\,\ddr t
  =\int_0^x\bigg(\frac{t^2}2+O(1)\bigg)\ddr x
  =\frac{x^3}6+O(x)
  $$
and, in turn,
  $$m(x)=\frac16\Big(\log\frac1x\Big)^3
  +O(\log x).$$    
We look now for the auxiliary function $\ell(x)$ of the slowly varying function $m(x)$. We have
\begin{align*}
  6\big(m(\lambda x)-m(x)\big)
  &=-(\log x+\log\lambda)^3+(\log x)^3 
  +O(\log x)\\
  &=-3\log\lambda(\log x)^2+O(\log x)
\end{align*}
so that, as $x\to 0$
  $$\frac{m(\lambda x)-m(x)}
  {(\log x)^2}
  \to -\frac12\log\lambda.$$
Hence, the auxiliary function of $\overrightarrow\nu(x)$ is found to be
  $\ell(x)=(\log x)^2$
with $c=-\frac12$, cf. \eqref{eq:deHaan}. Note that, because of the cancellation of the $(\log 1/x)^2$ term in the expansion of $m(x)$ for $x\to\infty$, the derivation of $\ell(x)$ only requires the leading term of $m(x)$. The latter can be alternatively obtained by using regular variation theory. Since
  $\int_0^xt\fprime(t)\ddr t$
is regularly varying at infinity with index $2$, Karamata's theorem yields
  $\int_0^x(x-t)t\fprime(t)\ddr t\big/
  x\int_0^x t\fprime(t)\ddr t
  \to\frac13$
as $x\to\infty$. A second application of Karamata's Theorem yields
  $\int_0^x t\fprime(t)\ddr t\big/
  x^2\fprime(x)\to\frac12$
to conclude that, as $x\to\infty$,
  $$\frac{\int_0^x(x-t)t\fprime(t)\ddr t}
  {x^3\fprime(x)}
  \to\frac12\frac13=\frac16$$
Since $\fprime(x)\to 1$ as $x\to\infty$, we get $m(x)\sim\frac16(\log 1/x)^3$.
Finally, by applying Theorem \ref{th:2nd-order} we conclude that
\begin{equation*}
  \E(K_n)=\frac16 (\log n)^3
  +\frac12 \gamma(\log n)^2+o(\log^2 n)
\end{equation*}  
It is worth stressing that $\pi(p)=-\log p$ yields $\E(p)=1/4$, i.e. a mass shift to lower values of $p$ compared to $\pi(p)=1$. This, according to $\overrightarrow\nu(x,p)\sim \log x/\log(1-p)$, favors larger values of $\E(K_n|p)$, which explain a faster growth of $\E(K_n)$.
\end{example}
The asymptotic expansion of $\E(K_n)$, for $m$ any positive integer, is derived in the following theorem. The key ingredient consists in expressing $\int_0^x t^m \fprime(t)\ddr t$ in terms of fractional integrals of $F(x)$.
%
%
\begin{theorem}\label{th:pier2}
Let $p$ in \eqref{eq:geom} have distribution $\pi(p)$ defined by $p\stackrel{d}{=}\edr^{-X}$ with $X\sim\text{\rm ga}(m+1,1)$ and $m$ a positive integer. Then
\begin{align*}
  \overrightarrow\nu(x)
  &=\frac{(\log 1/x)^{m+2}}{(m+2)!}
  +O\big((\log x)^m),\quad x\to 0\\
  \E(K_n)
  &=\frac{(\log n)^{m+2}}{(m+2)!}
  +\gamma\frac{(\log n)^{m+1}}{(m+1)!}
  +o((\log n)^{m+1}),\quad n\to\infty.
\end{align*}
\end{theorem}
%
%
\begin{proof}
Since $\pi(p)=(-\log p)^m/m!$, $m(x)$ in \eqref{eq:m(x)} becomes
\begin{equation}\label{eq:m(x)4}
  m(x)
  =\int_0^{\log 1/x}(\log 1/x-t)
  \frac{t^m}{m!}\fprime(t)\ddr t 
  =\int_0^{\log 1/x}\int_0^t
  \frac{s^m}{m!}\fprime(s)\ddr s\, \ddr t.
\end{equation}  
where the second equality follows again by integration by parts. Recall that, for an integer-valued index, fractional integrals correspond to \comillas{higher order} primitives of $\fprime(x)$:
  $F(x)={}_0F(x)$,
  ${}_1F(x)=\int_0^xF(t)\ddr t$
and  
  $${}_{(k+1)}F(x)=\int_0^x(x-t)\ddr\, {}_kF(t)
  =\int_0^x{}_kF(t)\ddr t,\quad k=1,2,\ldots$$
By repeated integration by parts the inner integral in \eqref{eq:m(x)4} is found to be
\begin{align*}
  \int_0^x
  \frac{t^m}{m!}\fprime(t)\ddr t
  &=\sum_{k=0}^m
  \frac{(-1)^k}{(m-k)!}x^{m-k}{}_kF(x).
\end{align*}
Next, we exploit the asymptotic evaluation \eqref{eq:F(x)}, $F(x)=x-\gamma+O(\edr^{-x})$ as $x\to\infty$, to find
\begin{align*}
  {}_1F(x)&=\frac{(x-\gamma)^2}2+O(1)
  =\frac{x^2-2\gamma x}2+O(1)\\
  {}_2F(x)&=\frac{(x-\gamma)^3}{3!}+O(x)
  =\frac{x^3-3\gamma x^2}{3!}+O(x)\\
  {}_kF(x)&=\frac{(x-\gamma)^{k+1}}{(k+1)!}+O(x^{k-1})
  =\frac{x^{k+1}-(k+1)\gamma x^k}{(k+1)!}+O(x^{k-1}).
\end{align*}
We get
\begin{align*}
  \int_0^x\frac{t^m}{m!}\fprime(t)\ddr t
  &=\sum_{k=0}^m\frac{(-1)^k}{(m-k)!}
  x^{m-k}\bigg(
  \frac{x^{k+1}-(k+1)\gamma x^k}{(k+1)!}+O(x^{k-1})
  \bigg)\\
  &=x^{m+1}\sum_{k=0}^m\frac{(-1)^k}{(m-k)!(k+1)!}
  +\gamma x^m\sum_{k=0}^m\frac{(-1)^{k+1}}{(m-k)!k!}
  +O(x^{m-1}).
\end{align*}
As for the $x^{m+1}$-term we obtain
\begin{align*}
  \sum_{k=0}^m\frac{(-1)^k}{(m-k)!(k+1)!}
  &=\sum_{k=1}^{m+1}\frac{(-1)^{k-1}}{(m+1-k)!(k)!}
  =\frac1{(m+1)!}
  \sum_{k=1}^{m+1}{m+1\choose k}(-1)^{k-1}\\
  &=-\frac1{(m+1)!}
  \bigg(\sum_{k=0}^{m+1}{m+1\choose k}(-1)^{k}
  -1\bigg)
  =\frac1{(m+1)!},
\end{align*}
where in the last step we used 
  $$\sum_{k=0}^{m+1}{m+1\choose k}(-1)^{k}
  =\sum_{k=0}^{m+1}{m+1\choose k}(-1)^{k}
  (+1)^{m+1-k}
  =(-1+1)^{m+1}=0.$$
A similar application of the binomial formula shows that the $x^m$-term is zero, namely
\begin{align*}
  \sum_{k=0}^m\frac{(-1)^{k+1}}{(m-k)!k!}
  &=-\sum_{k=0}^m\frac{(-1)^{k}}{(m-k)!k!}
  =-\frac{1}{m!}(-1+1)^m=0
\end{align*}
Hence, we have
  $$\int_0^x\frac{t^m}{m!}\fprime(t)\ddr t
  =\frac{x^{m+1}}{(m+1)!}+O(x^{m-1}),$$
which yields
  $$\int_0^x\int_0^t\frac{s^m}{m!}\fprime(s)
  \ddr s\,\ddr t
  =\int_0^x \bigg(
  \frac{t^{m+1}}{(m+1)!}+O(t^{m-1})\bigg)\ddr t
  =\frac{x^{m+2}}{(m+2)!}+O(x^m)$$
to conclude that, as $x\to 0$,
  $$m(x)=\frac{(\log 1/x)^{m+2}}{(m+2)!}
  +O\big((\log x)^m\big).$$
In view of $m(x)\leq \overrightarrow\nu(x)\leq m(x)+1$, the first statement is proved. We now proceed to derive the auxiliary function $\ell(x)$ of $m(x)$ and, in turn, of $\overrightarrow\nu(x)$. We have
\begin{align*}
  (m+2)!\big(m(\lambda x)-m(x)\big)
  &=(-\log(\lambda x))^{m+2}-(-\log x)^{m+2}
  +O\big((\log x)^m\big)\\
  &=(-1)^{m+2}\big( (\log x+\log\lambda)^{m+2}
  -(\log x)^{m+2}\big)+O\big((\log x)^m\big)\\
  &=(-1)^{m+2}(m+2)\log\lambda(\log x)^{m+1}
  +O\big((\log x)^m\big)
\end{align*}
so that, as $x\to 0$,
\begin{align*}
  \frac{m(\lambda x)-m(x)}
  {(\log x)^{m+1}}
  \to \frac{(-1)^{m+2}}{(m+1)!}\log\lambda.
\end{align*}
So we find
  $\ell(x)=(\log(x))^{m+1}$ and
  $c=(-1)^{m+2}/(m+1)!$
in \eqref{eq:deHaan}. Note that
\begin{align*}
  c\ell(1/n)
  &=\frac{(-1)^{m+2}}{(m+1)!}(-\log n)^{m+1}
  =\frac{(-1)^{m+2+m+1}}{(m+1)!}(\log n)^{m+1}
  =-\frac{(\log n)^{m+1}}{(m+1)!}.
\end{align*}
Finally, an application of Theorem \ref{th:2nd-order} yields the second statement.
\end{proof}
%
%
\begin{remark}
The phenomenon observed in Example \ref{example:1} for $m=1$, namely  the cancellation of the term $(\log 1/x)^2$ in the expansion of $m(x)$, applies to any $m\geq1$ meaning that the term $(\log 1/x)^{m+1}$ cancels out. Since the auxiliary function $\ell(x)$ is found to be of the order $(\log 1/x)^{m+1}$, we conclude that for any $m\geq 1$ the leading term in the expansion of $m(x)$ is sufficient for the derivation of the second order term in the expansion of $\E(K_n)$ according to Theorem \ref{th:2nd-order}. As observed in Example \ref{example:1}, a double application of Karamata's Theorem yields the leading term of $m(x)$ as it yields, for $x\to\infty$,
  $$\frac{\int_0^x(x-t)
  t^m\fprime(t)\ddr t}
  {x^{m+2}\fprime(x)}
  \to\frac1{m+1}\frac1{m+2}.$$
These calculations can be easily extended to the case of $p\stackrel{d}{=}\edr^{-X}$ for $X\sim\text{ga}(1+\rho,1)$ with $\rho>-1$. Since $\pi(\edr^{-t})\fprime(t)$ is regularly varying at infinity with index $\rho+1>0$,
  $$\frac{\int_0^x(x-t)
  \pi(\edr^{-t})\fprime(t)\ddr t}
  {x^{\rho+2}\fprime(x)}
  \to \frac1{\Gamma(\rho+1)}
  \frac1{\rho+1}\frac1{\rho+2}
  =\frac1{\Gamma(\rho+3)}$$
so we obtain
  $$\E(K_n)\sim \frac1{\Gamma(\rho+3)}
  (\log n)^{\rho+2}.$$
However, a more accurate approximation of 
  $\int_0^x(x-t)
  \pi(\edr^{-t})\fprime(t)\ddr t$
is necessary in order to apply Theorem \ref{th:2nd-order} and obtain the second order term in the expansion of $\E(K_n)$.
\end{remark}

\section{Negative binomial extension}
\label{sec:4}

In the following we use the notation $w_j(p)$ for the $j$th frequency as a function of the parameter $p$. Recall that the asymptotic behavior of $\E(K_n)$ depends on the behavior at zero of the tail mean measure
  $\overrightarrow\nu(x)=\int_0^1\overrightarrow\nu(x,p)\pi(p)\ddr p,$
where $\pi(p)$ is the success probability distribution, and
  $\overrightarrow\nu(x,p)=\#\{j:w_j(p)\geq x\}$
is the number of frequencies larger than a threshold $x\in[0,1]$. When the frequencies $w_j(p)$ are decreasing,  $\overrightarrow\nu(x,p)=\sup\{j:w_j(p)\geq x\}$, that is $\overrightarrow\nu(x,p)$ is obtained in terms of 
the inverse of $w_j(p)$ with respect to $j$. In the case of geometric frequencies the inverse is explicitly found to be $\log(x/p)/\log (1-p)+1$, thus we have 
  $\overrightarrow\nu(x,p)=\lfloor\log(x/p)/\log (1-p)+1\rfloor$
for $p\geq x$ or, equivalently, for $w_1(p)\geq x$. In Section \ref{sec:3}, the behavior in zero of $\overrightarrow\nu(x)$ was studied in terms of 
  $$m(x)=\int_0^1
  \frac{\log(x/p)}{\log (1-p)}
  \Indic_{(w_1(p)\geq x)}\pi(p)\ddr p,
  $$
based on the fact that
  $m(x)\leq \overrightarrow\nu(x)
  \leq m(x)+1$. 
  
In this section we consider a different model for $w_j(p)$ that can be seen as an extension of the geometric case. 
To this aim, we resort to a derivation of the geometric weights $w_j(p)=p(1-p)^{j-1}$ as a special case of a general construction of distributions on the positive integers with decreasing frequencies. Let $\phi(r;p)$ be a probability function for $r=1,2,\ldots$ with parameter $p\in(0,1)$. 
Then 
\begin{equation*}
  w_j(p)=\sum_{r\geq j}\frac{\phi(r;p)}{r},\
  \quad j=1,2,\dots,
\end{equation*}
form a decreasing sequence, $w_j(p)>w_{j+1}(p)$, of positive numbers summing up to one. As such, $(w_j(p))_{j\geq 1}$ defines a new distribution on the positive integers parametrized by $p$. 
An interesting instance of $\phi(r;p)$ is given by 
  $$\phi(r;p)=\phi(r;s,p)
  ={r+s-2\choose r-1}p^s(1-p)^{r-1},\quad r=1,2,\ldots$$ that is the negative binomial distribution shifted by one.  
The geometric frequencies are obtained by taking the  scale parameter $s=2$. In fact 
\begin{align*}
  w_j(p)&=\sum_{r\geq j}p^2(1-p)^{r-1}
  =p^2(1-p)^{j-1}\sum_{i\geq 1}(1-p)^{i-1}\\
  &=p^2(1-p)^{j-1}\sum_{i\geq 0}(1-p)^{i}
  =p(1-p)^{j-1}.
\end{align*}
An explicit expression for $w_j(p)$ can be found for any integer $s\geq 2$. 
When $s=3$, 
differentiating with respect to the geometric series one finds
\begin{align*}
  \frac{2(1-p)}{p^3}w_j(p)
  &=\sum_{r\geq j}(r+1)(1-p)^{r}
  =-{\ddr\over\ddr p}\sum_{r\geq j}(1-p)^{r+1}
  =-{\ddr\over\ddr p}\frac{(1-p)^{j+1}}{p}\\
  &
  =\frac{(1-p)^{j}}{p^2}\big((j+1)p+(1-p)\big)
  =\frac{(1-p)^{j}}{p^2}\big(1+jp\big)
\end{align*}
to conclude that
\begin{equation}\label{eq:wj}  
  w_j(p)=p(1-p)^{j-1}\frac{1+jp}{2},\quad j=1,2,\ldots
\end{equation}  
Similar formulae are derived for $s>3$: one finds that $w_j(p)$ is proportional to the geometric probability $p(1-p)^{j-1}$ multiplied by a polynomial in $(j,p)$ of order determined by $s$. Details are omitted. For instance, $s=4$ yields
  $$w_j(p)=p(1-p)^{j-1}
  \frac{2+2jp+jp^2+j^2p^2}{6},\quad j=1,2,\ldots$$

With a closed form expression of $w_j(p)$, an asymptotic evaluation of $\overrightarrow\nu(x,p)$ can be derived for $x\to 0$, and in turn, for $\overrightarrow\nu(x)$, so that the asymptotics of $\E(K_n)$ is obtained through Theorem \ref{th:2nd-order}. Let us restrict attention to $s=3$ with $w_j(p)$ as in \eqref{eq:wj} and $\pi(p)$ the uniform distribution. Our goal is to investigate the asymptotics of $\E(K_n)$ in comparison with the geometric case of Proposition \ref{prop:pier}. It is reasonable to expect that $\E(K_n)$ grows at a faster rate: in fact, the larger $s$, the larger the mean of the negative binomial distribution $\phi(r;s,p)$, so $w_j(p)$ decrease slower in $j$ for $s=3$ compared to $s=2$, which corresponds to the geometric case. This implies that $\overrightarrow\nu(x,p)$ grows slower as $x\to 0$ for $s=3$ and, in turn, via a Tauberian theorem $\E(K_n)$ grows faster as $n\to \infty$. In Proposition \ref{prop:pier3} we establish that the asymptotic behavior of $\E(K_n)$ differs from the one found in Proposition \ref{prop:pier} only in the second order term of the expansion.
%
%
\begin{proposition}\label{prop:pier3}
Let $w_j(p)$ be defined as in \eqref{eq:wj} and $p$ be uniformly distributed on $(0,1)$.
Then
\begin{align*}
  \overrightarrow\nu(x)&=
  \frac{1}{2}\Big(\log\frac1x\Big)^2
  +\log\frac1x\log\log\frac1x
  -\gamma\log\frac1x
  -(1+\log 2)\log\frac1x
  + O\bigg(\log\log\frac1x\bigg),
  \quad x\to 0\\
  \E(K_n)
  &=\frac{1}{2}\big(\log n\big)^2
  +\log n\log(\log n)
  -(1+\log 2)\log n
  +o(\log n),\quad n\to\infty.
\end{align*}
\end{proposition}
%
%
\begin{proof}
Let $m(x,p)\geq 0$ be defined by
\begin{equation}\label{eq:m(x,p)}
  p(1-p)^{m(x,p)}\frac{1}{2}\big(1+p+p\,m(x,p)\big)=x,
\end{equation}  
which corresponds to the solution in $m$ to the equation $w_{m+1}(p)=x$. Note that $m(x,p)\geq 0$ when $w_1(p)\geq x$, 
  $\overrightarrow\nu(x,p)=\lfloor m(x,p)+1 
  \rfloor \Indic_{(w_1(p)\geq x)}$
and, as in the geometric case,
  $m(x)\leq \overrightarrow\nu(x)\leq m(x)+1$
for  
\begin{equation*}
  m(x)=\int_0^1m(x,p)\Indic_{(w_1(p)\geq x)}\ddr p.
\end{equation*}  
Equation \eqref{lemma:m(x,p)} provides an asymptotic expansion of $m(x,p)$ as $x\to 0$. The proof is reported in the \hyperref[appn]{Appendix} and involves the Lambert W function \citep{Cor:etal:96}. 
\begin{multline}\label{lemma:m(x,p)}
  m(x,p)
  =\frac{\log x/p}{\log(1-p)}\\
  -\frac{1}{\log(1-p)}
  \log\bigg(
  \frac{1}{2}\bigg(1+p+p\frac{\log x/p}{\log(1-p)}
  +\frac{p}{\log(1-p)}\log\frac{-2\log(1-p)}{p}
  \bigg)\bigg)\\
  +\frac1{\log(1-p)}
  O\bigg(\frac{\log(\log1/x)}{\log 1/x}\bigg).
\end{multline}
An heuristic derivation 
inspired by \cite[Example 3.13]{Bar:Cox:89} is as follows. In equation \eqref{eq:m(x,p)}, we have that for small $x$, $m(x,p)$ will be large and the term 
  $p(1-p)^{m(x,p)}$ is
thus dominant. Rewrite the equation after taking the log and keeping $m(x,p)$ on the left hand side,
  $$m(x,p)=\frac{\log x/p}{\log(1-p)}
  -\frac{1}{\log(1-p)}\log\bigg(
  \frac{1}{2}\big(1+p+pm(x,p)\big)\bigg).$$ 
It defines a convergent iterative scheme via
  $$m_{(k)}(x,p)=\frac{\log x/p}{\log(1-p)}
  -\frac{1}{\log(1-p)}\log\bigg(
  \frac{1}{2}\big(1+p+pm_{(k-1)}(x,p)\big)\bigg)$$ 
with $m_{(1)}(x,p)$ solution to $p(1-p)^{m(x,p)}=x$, that is
  $$m_{(1)}(x,p)=\frac{\log x/p}{\log(1-p)}.$$
For $k=2$ we get
  $$m_{(2)}(x,p)=\frac{\log x/p}{\log(1-p)}
  -\frac{1}{\log(1-p)}\log\bigg(
  \frac{1}{2}\bigg(1+p+p\frac{\log x/p}{\log(1-p)}
  \bigg)\bigg),$$  
which nearly matches the expansion in \eqref{lemma:m(x,p)}.
Now we use it
to evaluate the behavior of $m(x)$ and, in turn, $\overrightarrow\nu(x)$ as $x\to 0$. Note that
  $$m(x)=\int_0^1m(x,p)\Indic_{(w_1(p)\geq x)}\ddr p
  =\int_{\tilde{x}}^1m(x,p)\ddr p,$$
where $\tilde{x}$ is defined by $w_1(\tilde x)=x$, that is 
  $\tilde{x}(1+\tilde{x})/2=x$.
It is easy to check that $x\leq\tilde{x}\leq 2x$ for any $x$ and $\tilde{x}\sim 2x$ as $x\to 0$. 
In order to exploit the derivation of the asymptotic expansion of $m(x)$ as $x\to 0^+$ in the geometric case, cf. proof of Proposition \ref{prop:pier}, we use \eqref{lemma:m(x,p)}
as follows:
  $$m(x)=\int_x^1\frac{\log x/p}{\log(1-p)}\ddr p
  -\int_{x}^{\tilde{x}}
  \frac{\log x/p}{\log(1-p)}\ddr p
  +\int_{\tilde{x}}^1\bigg(m(x,p)
  -\frac{\log x/p}{\log(1-p)}
  \bigg)\ddr p.$$
From Proposition \ref{prop:pier}, as $x\to0$ 
  $$\int_x^1\frac{\log x/p}{\log(1-p)}\ddr p
  =
  \frac{1}{2}\Big(\log\frac1x\Big)^2
  -\gamma\log\frac1x+O(1).$$
The first statement of the thesis about the behavior of $\overrightarrow\nu(x)$ as $x\to 0$ is then implied by $m(x)\leq \overrightarrow\nu(x)\leq m(x)+1$ and by showing that, as $x\to 0$,
\begin{align}
  \label{eq:pier2}
  \int_{x}^{\tilde{x}}
  \frac{\log x/p}{\log(1-p)}\ddr p&=O(1)\\
  \label{eq:pier3}
  \int_{\tilde{x}}^1\bigg(m(x,p)
  -\frac{\log x/p}{\log(1-p)}
  \bigg)\ddr p
  &=\log\frac1x\log\log\frac1x
  -(1+\log 2)\log\frac1x
  + O\bigg(\log\log\frac1x\bigg).
\end{align}
As for \eqref{eq:pier2}, 
the maximum of $(\log x/p)/\log(1-p)$ is attained at $p=p(x)$, where $p(x)$, the solution to the first order equation in $p$
  $-(1-p)\log(1-p)+p\log x/p=0$,
goes to zero as $x\to0$.
It can be shown that that  $p(x)\geq 2x$ for $x\leq 1/4$, that is $(\log x/p)/\log(1-p)$ is increasing for $x\leq p\leq2x$ and $x$ sufficiently small.
Since $2x\geq \tilde{x}$,
  $$\int_x^{\tilde{x}}
  \frac{\log x/p}{\log(1-p)}\ddr p
  \leq\int_x^{2x}
  \frac{\log x/p}{\log(1-p)}\ddr p
  \leq x\frac{\log x/(2x)}{\log(1-2x)}
  =-x\log2/\log(1-2x)\leq \log 2/2$$
so \eqref{eq:pier2} follows.
As for \eqref{eq:pier3}, note that by the change of variable $t=\log 1/p$,
  $$\int_{\tilde{x}}^1-\frac{1}{\log(1-p)}\ddr p
  =\int_0^{\log1/\tilde{x}}f(t)\ddr t
  =F(\log1/\tilde{x}),$$
for $\fprime(t)$ in \eqref{eq:m(x)} and $F(t)$ the primitive of $\fprime(t)$. By equation \eqref{lemma:m(x,p)},
\eqref{eq:F(x)} and $\tilde x\sim 2x$ as $x\to0$, we have that, as $x\to 0$,
\begin{multline*}
  \int_{\tilde{x}}^1
  \bigg(m(x,p)-\frac{\log x/p}{\log(1-p)}\bigg)\ddr p
  =-\log2\log1/x+ O(\log(\log 1/x))\\
  +\int_{\tilde{x}}^1
  -\frac{1}{\log(1-p)}
  \log\bigg(
  1+p
  +p\frac{\log x/p}{\log(1-p)}
  +\frac{p}{\log(1-p)}
  \log\frac{-2\log(1-p)}{p}
  \bigg)\ddr p.
\end{multline*}
In studying the asymptotic behavior of the integral in the last display, it is sufficient to focus on  
  $$\int_{\tilde x}^1
  -\frac{1}{\log(1-p)}
  \log\bigg(
  1+p\frac{\log x/p}{\log(1-p)}
  \bigg)\ddr p,$$
since the extra terms inside the logarithm satisfy
  $$-2\edr^{-1}
  \leq p+\frac{p}{\log(1-p)}\log\frac{-2\log(1-p)}{p}
  \leq 1$$ 
for $0\leq p\leq1$. By the change of variable $t=\log 1/p$
  $$\int_{\tilde x}^1
  -\frac{1}{\log(1-p)}
  \log\bigg(
  1+p\frac{\log x/p}{\log(1-p)}
  \bigg)\ddr p
  =\int_0^{\log 1/\tilde{x}}
  \log\Big(1+\big(\log 1/\tilde{x} - t\big)
  \fprime(t)\Big)\,\fprime(t)\ddr t,$$
for $\fprime(t)$ defined in \eqref{eq:m(x)}. Hence, \eqref{eq:pier3} is implied by 
\begin{equation}\label{eq:pier4}
  r(x)=
  \int_0^{x}
  \log\big(1+\big(x - t\big)
  \fprime(t)\big)\,\fprime(t)\ddr t
  =x\log(x)-x+O(\log x),
\end{equation}
as $x\to \infty$. We have
\begin{align*}
  r(x)&=
  \int_0^{x}\bigg(
  \log x+\log f(t)
  +\log\frac{1+\big(x - t\big)
  \fprime(t)}{x\fprime(t)}\bigg)\fprime(t)\ddr t\\
  &=\log(x)F(x)
  +\int_0^x \fprime(t)\log\fprime(t)\ddr t
  +\int_0^x\log\bigg(
  1+\frac{1-t\fprime(t)}{x\fprime(t)}
  \bigg)\fprime(t)\ddr t.
\end{align*}
The first term on the right hand side is
  $\log x(x-\gamma+O(\edr^{-x})),$
as $x\to\infty$, due to the asymptotic expansion of $F(x)$ in \eqref{eq:F(x)}. The second term is easily shown to be bounded in absolute value uniformly in $x$. As for the third term, we are left to show that
  $$\int_0^x\log\bigg(
  1+\frac{1-t\fprime(t)}{x\fprime(t)}
  \bigg)\fprime(t)\ddr t
  =-x+\log x+O(1),$$ 
as $x\to \infty$.
To this aim, it is convenient to split the integral as
\begin{equation}\label{eq:split}
  \int_0^1\log\bigg(
  1+\frac{1-t\fprime(t)}{x\fprime(t)}
  \bigg)\fprime(t)\ddr t+
  \int_1^x\log\bigg(
  1+\frac{1-t\fprime(t)}{x\fprime(t)}
  \bigg)\fprime(t)\ddr t.
\end{equation}  
The first integral in \eqref{eq:split} is bounded in $x$ since 
  $$\int_0^1\log\bigg(
  1+\frac{1-t\fprime(t)}{x\fprime(t)}
  \bigg)\fprime(t)\ddr t
  \leq\frac1x
  \int_0^1\big(1-t\fprime(t)\big)
  \ddr t,$$
whereas the last integral is a positive and finite constant. As for the second integral in \eqref{eq:split}, since $1-f(t)\sim\edr^{-t}/2$ for $t\to\infty$, cf. Lemma \ref{lemma:fprime}, 
  $$\int_1^x\log\bigg(
  1+\frac{1-t\fprime(t)}{x\fprime(t)}
  \bigg)\fprime(t)\ddr t
  \sim
  \int_1^x\log\bigg(
  1+\frac{1-t}{x}
  \bigg)\ddr t$$
as $x\to \infty$ and
  $$\int_1^x\log\bigg(
  1+\frac{1-t}{x}
  \bigg)\ddr t=-x+\log x+1$$
by direct calculation. Hence \eqref{eq:pier4}, and in turn \eqref{eq:pier3}, follow. The proof of the first statement of the proposition is then complete.
The second statement about the expansion of $\E(K_n)$, as $n\to\infty$, follows from an application of Theorem \ref{th:2nd-order}. 
\end{proof}

\begin{appendix}

\section*{Appendix}\label{appn}

\subsection*{Proof of Theorem \ref{th:2nd-order}}

The proof follows arguments similar to those of \cite[Theorem 3.9.1]{Bin:Gol:Teu:87}. It consists in evaluating 
  $\big(\Phi(n)-\overrightarrow\nu(1/n)\big)/\ell(1/n)$
in the decomposition
  $$\E(K_n)
  =\overrightarrow\nu(1/n)
  +\frac{\Phi(n)-\overrightarrow\nu(1/n)}{\ell(1/n)}\ell(1/n)
  +\E(K_n)-\Phi(n).$$
Indeed, as $\Phi(n)\sim\overrightarrow\nu(1/n)$ and $\ell(1/n)$ are slowly varying, 
  $|\E(K_n)-\Phi(n)|
  \leq \Small{\frac2n}\Phi(n)=o(\ell(1/n))$,
cf. \eqref{eq:Poissonization}, so the conclusion follows by showing that
  $\big(\Phi(n)-\overrightarrow\nu(1/n)\big)/\ell(1/n)
  \to-c\gamma$.
To this aim,
\begin{align*}
  \frac{\Phi(1/x)-\overrightarrow\nu(x)}{\ell(x)}
  &=\frac1{\ell(x)}\bigg[
  \int_0^\infty\frac1x\edr^{-y/x}\overrightarrow\nu(y)\ddr y
  -\int_0^\infty\overrightarrow\nu(x)\edr^{-\lambda}\ddr \lambda
  \bigg]\\
  &=\frac1{\ell(x)}\bigg[
  \int_0^\infty\edr^{-\lambda}\overrightarrow\nu(\lambda x)
  \ddr \lambda
  -\int_0^\infty\overrightarrow\nu(x)\edr^{-\lambda}\ddr \lambda
  \bigg]
  =\int_0^\infty
  \frac{\overrightarrow\nu(\lambda x)-\overrightarrow\nu(x)}{\ell(x)}
  \edr^{-\lambda}\ddr \lambda\\
  &\to\int_0^\infty c(\log\lambda)\edr^{-\lambda}
  \ddr\lambda
  =c\Gamma'(1)=-c\gamma,\quad\text{as }x\to 0,
\end{align*} 
where in taking the limit we used the dominated convergence theorem, cf. global bounds in \cite[Theorem 3.8.6]{Bin:Gol:Teu:87}. 
\hfill \qed

\subsection*{Proof of Lemma \ref{lemma:fprime}}

We will use the following integral representation of the Euler-Mascheroni constant: 
  $$\gamma=\int_0^\infty\bigg( \frac{1}
  {1-\edr^{-x}}-\frac{1}{x}
  \bigg)\edr^{-x}\ddr x.$$
By the change of variable $t=-\log(1-\edr^{-x})$ so that
  $\ddr t=-\frac{\edr^{-x}}{1-\edr^{-x}} \ddr x$
and
  $x=-\log(1-\edr^{-t})$,
we obtain
\begin{align*}
  \gamma
  &=\int_0^\infty\bigg( \frac{1}
  {1-\edr^{-x}}-\frac{1}{x}
  \bigg)\edr^{-x}\ddr x
  =\int_0^\infty\bigg(1-\frac{1-\edr^{-x}}{x}
  \bigg)\frac{\edr^{-x}}{1-\edr^{-x}}\ddr x\\
  &=\int_0^\infty\bigg(1-\frac{\edr^{-t}}
  {-\log(1-\edr^{-t})}
  \bigg)\ddr t
  =\int_0^\infty \big(1-\fprime(t)\big)\ddr t.
\end{align*}
It is easy to check that $\lim_{t\to 0}\fprime(t)=0$ and $\lim_{t\to\infty}\fprime(t)=1$.
As for the tail behavior, by the Taylor expansion of $\log(1+x)=x-x^2/2+O(x^3)$ as $x\to 0$, we find that, as $t\to\infty$,
\begin{align*}
  1-\fprime(t)
  &=1-\frac{\edr^{-t}}{-\log(1-\edr^{-t})}
  \sim 1-\frac{\edr^{-t}}{\edr^{-t}+\edr^{-2t}/2}
  =\frac{\edr^{-2t}/2}{\edr^{-t}+\edr^{-2t}/2}
  \sim\frac{\edr^{-t}}{2}.
\end{align*}
\hfill \qed

\subsection*{Proof of Equation \eqref{lemma:m(x,p)}}

Let $W(z)$ be the Lambert function defined by
  $W(z)\edr^{W(z)}=z,$
where $W(z)$ is a multivalued function that has, for $z$ a real number, two branches, the principal branch $W_0(z)$ for $W(z)\geq -1$, and the branch $W_{-1}(z)$ for $W(z)< -1$. We have that $\lim_{z\to0^+}W_0(z)=0$ while $\lim_{z\to 0^-}W_{-1}(z)=-\infty$.
In particular, according to \cite[Section 4]{Cor:etal:96}, as $z\to 0^-$
\begin{equation}\label{eq:LambertW}
  W_{-1}(z)=
  \log(-z)-\log(-\log(-z))+O\bigg(
  \frac{\log(-\log(-z))}{\log(-z)}\bigg).
\end{equation}
By algebraic manipulation of \eqref{eq:m(x,p)} 
\begin{align*}
  &p(1-p)^{m(x,p)}\frac{1}{2}(1+p+p\,m(x,p))=x;\quad
  \edr^{\log(1-p)m(x,p)}(1+p+p\,m(x,p))=2x/p;\\
  &(1+p+p\,m(x,p))\log(1-p)\edr^{\log(1-p)m(x,p)}
  =\frac{2x\log(1-p)}{p};\\
  &(1+p+p\,m(x,p))\log(1-p)\edr^{\log(1-p)(1/p+1+m(x,p)}
  =\frac{2x\log(1-p)}{p}\edr^{(1/p+1)\log(1-p)};\\
  &\frac{\log(1-p)}{p}(1+p+p\,m(x,p))
  \edr^{\frac{\log(1-p)}{p}(1+p+p\,m(x,p))}
  =\frac{2x\log(1-p)}{p^2}
  \edr^{\frac{1+p}{p}\log(1-p)};\\
  &\frac{\log(1-p)}{p}(1+p+p\,m(x,p))
  =W(z),
\end{align*}
where, in the last display,
\begin{equation}\label{eq:z}
  z=\frac{2x\log(1-p)}{p^2}
  \exp\Big(\frac{1+p}{p}\log(1-p)\Big).
\end{equation}  
Solving for $m(x,p)$,
\begin{equation}\label{eq:T(x)_2}
  m(x,p)=\frac{1}{p\log(1-p)}
  \bigg(pW_{-1}(z)
  -\log(1-p)\bigg)-1,
\end{equation}  
where we used the branch $W_{-1}$ of $W(z)$ since $z$ in \eqref{eq:z} is $\leq 0$ and $W(z)\leq-1$. The fact that $W(z)\leq-1$ is easily checked by using $m(x,p)\geq 0$. In fact
\begin{align*}
  &\frac{1}{p\log(1-p)}\bigg(pW(z)
  -\log(1-p)\bigg)-1\geq 0;\quad
  pW(z)-\log(1-p)\leq p\log(1-p);\\
  &pW(z)\leq \log(1-p)(1+p);\quad 
  W(z)\leq \log(1-p)\frac{1+p}{p}
\end{align*}
and $\frac{1+p}{p}\log(1-p)$ decreases from $-1$ to $-\infty$ for $p\in(0,1)$. From \eqref{eq:z} one finds that $z\to 0^-$ as $x\to 0^+$. In particular, from $w_1(p)>x$, that is $p(1+p)/2>x$, it follows that
  $$\frac{1+p}{p}\log(1-p)
  \exp\Big(\frac{1+p}{p}\log(1-p)\Big)
  \leq z\leq 0$$
and the lower bound is larger than $-\edr^{-1}$ for any $p\in(0,1)$. 
Hence $\log(-z)<-1$ and $\log(-\log(-z))>0$. 
By direct calculation
  $$\log(-z)
  =\log(1-p)\bigg(
  \frac{\log x/p}{\log(1-p)}
  +\frac1{\log(1-p)}\log\frac{-2\log(1-p)}{p}
  +\frac{1+p}{p}\bigg)$$
and
  $$\frac{1}{p\log(1-p)}\bigg(p\log(-z)
  -\log(1-p)\bigg)-1
  =\frac{\log x/p}{\log(1-p)}+
  \frac{1}{\log(1-p)}
  \log\frac{-2\log(1-p)}{p}.$$
Substitute in \eqref{eq:T(x)_2} $W_{-1}(z)$ for $\log(-z)-\log(-\log(-z))$ according to the two terms expansion in \eqref{eq:LambertW}, to find
\begin{align*}
  \frac{1}{p\log(1-p)}
  &\bigg(p\Big(\log(-z)
  -\log(-\log(-z))\Big)
  -\log(1-p)\bigg)-1\\
  &=\frac{1}{p\log(1-p)}
  \bigg(p\log(-z)-\log(1-p)\bigg)-1
  -\frac{1}{\log(1-p)}\log(-\log(-z))\\
  &=\frac{\log x/p}{\log(1-p)}+
  \frac{1}{\log(1-p)}
  \log\frac{-2\log(1-p)}{p}
  -\frac{1}{\log(1-p)}\log(-\log(-z))\\
  &=\frac{\log x/p}{\log(1-p)}
  -\frac{1}{\log(1-p)}
  \log\bigg(
  \frac{p}{2\log(1-p)}
  \log(-z)\bigg)\\
  &=\frac{\log x/p}{\log(1-p)}
  -\frac{1}{\log(1-p)}
  \log\bigg(
  \frac{p}{2}
  \bigg(
  \frac{\log x/p}{\log(1-p)}
  +\frac1{\log(1-p)}\log\frac{-2\log(1-p)}{p}
  +\frac{1+p}{p}\bigg)
  \bigg)\\
  &=\frac{\log x/p}{\log(1-p)}
  -\frac{1}{\log(1-p)}
  \log\bigg(\frac{1}{2}\bigg(
  1+p
  +p\frac{\log x/p}{\log(1-p)}
  +\frac{p}{\log(1-p)}
  \log\frac{-2\log(1-p)}{p}
  \bigg)\bigg).
\end{align*}
The remainder of the expansion is easily found. 
\hfill \qed

\end{appendix}




\bibliographystyle{apalike}
\bibliography{short-biblio}

\end{document}